\documentclass{amsart}
\usepackage[dvips]{graphicx}
\usepackage{amscd}
\usepackage{pstricks}

\theoremstyle{plain}
\newtheorem{theorem}{Theorem}
\newtheorem{proposition}{Proposition}

\newtheorem{definition}{Definition}
\newtheorem{corollary}{Corollary}

\theoremstyle{remark}
\newtheorem{remark}{Remark}

\numberwithin{equation}{section}

\newcommand{\abs}[1]{|#1|}

\begin{document}

\title[]{A Spectral Dichotomy Version of the Nonautonomous Markus--Yamabe Conjecture}
\author[]{\'Alvaro Casta\~neda}
\author[]{Gonzalo Robledo}
\address{Departamento de Matem\'aticas, Facultad de Ciencias, Universidad de
  Chile, Casilla 653, Santiago, Chile}
%\email{castaneda@uchile.cl, grobledo@uchile.cl}
\thanks{This research has been partially supported by MATHAMSUD program (16-MATH-04 STADE)}
\date{\today}
\begin{abstract}
In this article we introduce a nonautonomous version of the Markus--Yamabe conjecture from an exponential dichotomy spectrum point of view. We prove the validity of this conjecture for the scalar and triangular case. Additionally we show that the origin is a global attractor for an autonomous system by using nonautonomous dynamical systems tools.
\end{abstract}

\maketitle

\section{Introduction}

One of the central problems on autonomous and nonautonomous dynamical systems, is to determine
conditions under which certain sets $\Omega$ are ``attractors" for
some dynamics, that is, when the orbits of a set of points converge in the
future to $\Omega$. In the  autonomous continuous-time case, that is, flows
associated to vector fields $F\colon \mathbb{R}^{n}\to \mathbb{R}^{n}$, an infinitesimal condition that ensures
that an equilibrium point $x_{0}$ (\textit{i.e.}, $F(x_{0})=0$) is a local asymptotic attractor is that the Jacobian matrix at $x_0,$ namely,
$JF(x_{0})$ is Hurwitz, \textit{i.e.}, the real part of its eigenvalues is negative.

\subsection{Markus--Yamabe conjecture}

Motivated for this simple observation over the eigenvalues in the autonomous context in \cite{MY}, L. Markus $\&$
H. Yamabe establish their well known global stability conjecture.

\medskip

\noindent {\bf{The Markus--Yamabe Conjecture (MYC).}} Let $ F: \mathbb{R}^n
\to \mathbb{R}^n$ be a $C^1-$ vector field with $ F(0) = 0 $ such that $JF(x)$ is Hurwitz for any $ x
\in \mathbb{R}^n,$ then the origin is a global attractor of
the system \begin{equation}\label{nonlin}\dot{x} = F(x).\end{equation}

Let us recall that the vector fields satisfying the hypothesis of {\textbf{MYC}} are called Hurwitz vector fields. It is known that the MYC is true
when $ n \leq 2 $   and false when $ n \geq 3$ (see \cite{CEGMH} for a counterexample). The proofs in the planar context, both the polynomial case (G. Meisters and C. Olech in \cite{MO}) as the $ C^1-$ case (R. Fe{\ss}ler in \cite{F}, A.A. Glutsyuk in \cite{Glu} and C. Guti\'errez in \cite{Gu}) are based on a remarkable result of C. Olech \cite{O}, which proved that MYC is equivalent to the injectivity of the map $F$ when $n=2$.

Contrarily to the autonomous case, the eigenvalues analysis does not always allow any conclusion over the stability of the solutions in the nonautonomous framework. Indeed, L. Markus and H. Yamabe \cite{MY} constructed the counterintuitive example
\begin{displaymath}
\dot{x}=A(t)x \quad \textnormal{with} \quad
A(t)=\left(\begin{array}{cc}
-1+\frac{3}{2}\cos^{2}(t) & -1+\frac{3}{2}\cos(t)\sin(t)\\\\
-1-\frac{3}{2}\cos(t)\sin(t) & -1+\frac{3}{2}\sin^{2}(t)
\end{array}\right),
\end{displaymath}
which has eigenvalues with negative real part, namely the eigenvalues are $-1/4 \pm 1/4 \sqrt{7}i.$ Nevertheless, the system has  fundamental matrix

\begin{equation}
\label{Markus}
\Phi(t) = \left (
\begin{array}{rcl}
e^{\frac{1}{2} t} \cos t  & & e^{-t} \sin t\\
-e^{\frac{1}{2} t} \sin t &  & e^{-t} \cos t
\end{array}
\right ),
\end{equation}
which shows that there exists initial conditions $x(0),$ arbitrarily  close to the origin, for which the solutions escapes to infinity although we have the negativeness of the eigenvalues.

\subsection{A generalization to the nonautonomous context}
We have seen that an eigenvalues--based approach has several shortcomings and is not an adequate tool to cope with stability issues in the nonautonomous framework. On the other hand, while in the autonomous context, the properties of asymptotical stability and hyperbolicity of an equilibrium are described in terms of the spectrum of the eigenvalues associated to the Jacobian matrix, in the nonautonomous framework several spectral theories are available (\emph{e.g}, Lyapunov, Bohl and Sacker $\&$ Sell) and we refer the reader to \cite{Dieci,Doan} for a deep discussion.

In order to generalize the MYC to the nonautonomous case, we will consider the property of exponential dichotomy which can be seen as the generalization of the hyperbolicity that allow the construction of a nonautonomous spectral theory namely, the Sacker $\&$ Sell's or exponential dichotomy spectrum, which is compatible with the uniform asymptotical stability.

\begin{definition}\cite{K}
\label{UAS}
Consider the nonlinear system
\begin{equation}
    \label{nolin1}
    \dot{x} = g(t,x)
\end{equation}
with $g(t,0) = 0$ for any $t \geq 0.$ The equilibrium point $x = 0$ is uniformly asymptotically stable  if there exist a $\mathcal{KL}$ function $\beta \colon \mathbb{R}^+ \times \mathbb{R}^+ \to \mathbb{R}^+$ and $c > 0$ with
$$|x(t)| \leq \beta(|x(t_0)|, t-t_0) \quad \forall t \geq t_0 \geq 0, \quad \textnormal{for all} \quad |x(t_0)| < c,$$
where $c$ is independent of $t_0$. We recall that $\beta\in \mathcal{KL}$ if $s \mapsto \beta(s,r)$ is strictly increasing with $\beta(0,r) = 0,$  for any fixed $r\geq 0$ and $r \mapsto \beta(s,r)$ is decreasing and tends to zero when $r \to \infty$ for any fixed $s$.
\end{definition}

\begin{definition}\cite{Coppel}
\label{DE}
The linear system
\begin{equation}
\label{lineal}
\dot{x} = A(t)x,
\end{equation}
where $t \mapsto A(t)$ is a locally integrable map, has an exponential dichotomy property on $J \subset \mathbb{R}$ if there exists a projection $P^{2}=P$
and constants $K\geq 1$, $\alpha>0$,  such that its fundamental matrix $\Phi(t)$ verifies:
\begin{equation}
\label{ED}
\left\{\begin{array}{rcl}
||\Phi(t)P\Phi^{-1}(s)||&\leq & Ke^{-\alpha(t-s)} \quad \textnormal{for any} \quad t\geq s, \quad t,s \in J,\\
||\Phi(t)(I-P)\Phi^{-1}(s)|| &\leq & Ke^{-\alpha(s-t)} \quad \textnormal{for any} \quad s \geq t, \quad t,s \in J.\\
\end{array}\right.
\end{equation}
\end{definition}

%It is easy to verify that the system %(\ref{Markus}) has an exponential dichotomy %on
%$\mathbb{R}$ with projector %$P=\left[\begin{array}{cc}
%0 & 0 \\
%0 & 1\end{array}\right]$.

%Associated to the hyperbolicity in this %nonautonomous framework, we have the %concept of exponential dichotomy spectrum %which is defined as follows

\begin{definition}\cite{Kloeden, SS}
The exponential dichotomy spectrum associated to (\ref{lineal}) is the set
\label{spectrum}
\begin{displaymath}
\Sigma_{J}(A)=\left\{\gamma \in \mathbb{R}\colon
\dot{x}=[A(t)-\gamma I]x \quad \textnormal{has not an exponential dichotomy on} \, J \right\}.
\end{displaymath}
\end{definition}

It is important to note that (\ref{lineal}) has an exponential dichotomy on $J$ if and only if $0\notin \Sigma_{J}(A)$, which prompt us to consider the exponential dichotomy as a generalization of the hyperbolicity property to the nonautonomous case.

Moreover, it is well known (see e.g. \cite[Th. 5.12]{Kloeden}) that if $||A(t)||$ is bounded on $J$  then
$$\Sigma_{J}(A)=[a_{1},b_{1}]\cup \cdots \cup [a_{\ell},b_{\ell}] \quad \textnormal{where} \quad \ell\leq n,$$
where $[a_i,b_i]$ are known as spectral intervals. It is interesting to point out that $\Sigma_{J}$ can be seen as a generalization of the eigenvalues spectrum since the spectral intervals play the role of the real parts of the eigenvalues.

In this work, we will be mainly focused in the particular case $J=[0,+\infty)$ and we will use the notation $\Sigma^{+}(A)$ for simplicity.

\begin{remark}
\label{observacion}
 There exists a strong relation between the three concepts above. Indeed,
 \begin{itemize}
     \item [(i)] The linear system (\ref{lineal}) is uniformly asymptotically stable (see e.g. \cite[Theorem 4.11]{K}) if and only if there exist constants $K \geq 1$ and $\alpha > 0$ such that
     $$|\Phi(t,t_0)| \leq K e^{-\alpha(t-t_0)} \quad \textnormal{for any} \quad t \geq t_0 \geq 0,$$ that is (\ref{lineal}) is uniformly asymptotically stable if and only if has exponential dichotomy on $\mathbb{R}^+$ with projector $P = I.$

     \item [(ii)] A trivial consequence of Gronwall's lemma is that the linear system (\ref{lineal}) is uniformly asymptotically stable if and only if $\Sigma^{+}(A) \subset (-\infty, 0).$ This fact mimics in some sense the notion of the Hurwitz matrix of the autonomous context.

     \item[(iii)] For any $\gamma \in (b_i, a_{i+1})$ the system $\dot{x} = [A(t)-\gamma I]x$ has an exponential dichotomy on $J$ with projector $P_{\gamma}$ whose rank is constant in this interval. We refer the reader to \cite[Lemma 5.11]{Kloeden} for details.
 \end{itemize}
\end{remark}

The roughness properties of the exponential dichotomy, namely, the preservation of the exponential dichotomy of \eqref{lineal} by a linear additive perturbation
\begin{equation}\label{robustez}
\dot{x} = [A(t) + B(t)]x
\end{equation}
such that $B(t)$ is small enough. The following results are recalled in order to make this article self contained.

\begin{proposition}\cite[p.34]{Coppel}
\label{RCoppel}
Assume that \eqref{lineal} has an exponential dichotomy on $[0, +\infty)$ with projector $P.$ If there exists a bounded continuous matrix $B(t)$ such that
$$
\displaystyle \sup_{t \in \mathbb{R}^+}||B(t)|| < \alpha / 4 K^2
$$ then the perturbed system \eqref{robustez}
also has an exponential dichotomy with projector $Q$ having the same null space has the projection $P.$
\end{proposition}
This result has been improved in several directions as in \cite{Coppel, Naulin, Palmer84, Wiggins}.  We will be focused in the following one

\begin{proposition}\cite[Corollary 3.1]{Wiggins}
\label{RWiggins}
Suppose \eqref{lineal} has an exponential dichotomy on $[T, +\infty)$ with projector $P.$ Suppose
$$
\limsup_{t \to \infty}||B(t)|| < \alpha / 2K,
$$
then \eqref{robustez} has an exponential dichotomy on $\mathbb{R}^+$ with projector $Q$ similar to $P.$
\end{proposition}

\section{The nonautonomous global stability conjecture}
 The goal of this article is to introduce the analogous to the MYC tailored to the nonautonomous case in terms of exponential dichotomy spectrum.

{\bf{Nonautonomous Markus--Yamabe Conjecture (NMYC):}} Let us consider the system
\begin{equation}
\label{MY}
\dot{x} = g(t,x)
\end{equation}
where $g\colon \mathbb{R}^{+}\times\mathbb{R}^{n}\to \mathbb{R}^{n}$ is such that
\begin{itemize}
\item[\textbf{(G1)}] $g$ is continuous with respect to $t$ and $C^1$ with respect to $x$.
\item[\textbf{(G2)}] $g(t,x)=0$ if and only if $x=0$ for all $t\geq 0$.
\item[\textbf{(G3)}] For any measurable function $t\mapsto y(t)$, the Jacobian matrix $Jg(t,y(t))$ is bounded.
\item[\textbf{(G4)}] For any measurable function $t\mapsto y(t)$, the family of linear systems
\begin{equation}
\label{MYNA}
\dot{z} = Jg(t,y(t))z
\end{equation}
has an exponential dichotomy spectrum satisfying
\begin{equation}
\label{SMY}
\Sigma^+(Jg(t,y(t))) \subset (-\infty,0),
\end{equation}
\end{itemize}
is the trivial solution globally uniformly asymptotically stable for (\ref{MY})?.
\bigskip

This can be seen as a generalization of MYC to the nonautonomous framework due to the strong analogy with the autonomous case: i) we have a linear system with ``spacial" variable coefficients, ii) the linear systems have exponential stability which are spectrally described for any measurable function $t\to y(t)$.

On the other hand, there is a completely different approach to this global stability problem given by D. Cheban in \cite[Theorem 4.4]{Cheban} which is based in the concept of cocycles and other tools of nonautonomous dynamical systems theory in order to find necessary and sufficient conditions for that the null section of a finite-dimensional vectorial bundle fiber to be  globally
uniformly asymptotically stable.

\section{Main Results}

%Let us consider the scalar equation
%\begin{equation}\label{escalar}
%\dot{x} = a(t)x.
%\end{equation}
%By statement (i) of Remark \ref{observacion} we %know that \eqref{escalar} is uniformly %asymptotically stable if and only if there exist %positive constants $K, \alpha$ such that

%$$\frac{1}{t-t_0} \int_{t_0}^{t} a(\tau) \, d %\tau \leq \frac{\ln K}{t-t_0} - \alpha \quad %\textnormal{for any} \quad t > t_0,$$
%which implies that the Bohl spectrum satisfies %$\Sigma_{Bohl}(a) \subset (-\infty, - \alpha).$

In this section we prove that the conjecture is true for $n=1$ and we study some particular cases in general dimensions which support this conjecture.

\begin{theorem}
\label{N1}
The \textnormal{\textbf{NMYC}} is verified for dimension $n=1$.
\end{theorem}

\begin{proof}
Let us consider the nonlinear system
\begin{equation}
\label{nlescalar}
\dot{x} = g(t,x),
\end{equation}
 with the initial condition $x_{0}\neq 0$ at time $t=t_{0}$. Without loss of generality, it will be supposed that $x_{0}>0$. By uniqueness of the solution it follows that $x(t)>0$ for any $t\geq t_{0}$.

Notice that any solution of \eqref{nlescalar} can be written as follows
$$
x(t) = x_{0} + \int_{t_0}^t g_2(\tau, \theta_{\tau})x(\tau) \, d \tau \quad \textnormal{where} \quad 0<\theta_{\tau} < x(\tau),
$$
where $g_2$ denotes the partial derivative with respect to the second variable. This identity is equivalent to
\begin{displaymath}
\begin{array}{rcl}
\displaystyle\frac{d}{dt}\ln\left(x_{0}+\int_{t_{0}}^{t}g_{2}(\tau,\theta_{\tau})x(\tau)\,d\tau\right) &=& g_{2}(t,\theta_{t}) \quad \textnormal{with $0<\theta_{t}< x(t)$}.
\end{array}
\end{displaymath}

We integrate between $t_{0}$ and $t$ obtaining that
\begin{displaymath}
x(t)=x_{0}\exp\left(\int_{t_{0}}^{t}g_{2}(\tau,\theta_{\tau})\,d\tau\right)\quad \textnormal{with $0<\theta_{\tau}< x(\tau)$},
\end{displaymath}
that is, the solutions of (\ref{nlescalar}) can be seen as solutions of the linear equation
\begin{equation}
\label{apro-nl}
\dot{z} = g_2(t,\theta_{t})z,
\end{equation}
for some measurable function $t\mapsto \theta_{t}$. Now, the assumption
\begin{displaymath}
\Sigma^{+}(g_{2}(t,\theta_{t}))\subset (-\infty,0)
\end{displaymath}
is equivalent (see Remark \ref{observacion}) to the uniform asymptotical stability of (\ref{apro-nl}).
\end{proof}

\begin{corollary}
Let us consider the triangular system
\begin{equation}
\label{triangular}
\begin{array}{rcl}
\dot{x}_{1}&=&g_{1}(t,x_{1},x_{2},\ldots,x_{n})\\
\dot{x}_{2}&=&g_{2}(t,x_{2},\ldots,x_{n})\\
 &\vdots&  \\
\dot{x}_{n}&=&g_{n}(t,x_{n}),
\end{array}
\end{equation}
whose right part, namely $G(t,x)$, verifies \textbf{(G1)}--\textbf{(G4)},
then the trivial solution of (\ref{triangular}) is globally uniformly asymptotically stable.
\end{corollary}

\begin{proof}
The assumption \textbf{(G4)} says that for any family of measurable functions $t\mapsto \theta_{i}(t)$ with $t\geq 0$ and $i=1,\ldots,n$, it follows that
\begin{displaymath}
\Sigma^{+}\left[JG(t,\theta_{1}(t),\ldots,\theta_{n}(t))
\right] \subset (-\infty,0),
\end{displaymath}
where the matrix $JG$ is defined by
\begin{displaymath}
JG(t,x_{1},\ldots,x_{n})_{ij}=\left\{
\begin{array}{ccl}
\frac{\partial g_{i}}{\partial x_{j}}(t,x_{i},\ldots,x_{n}) &\textnormal{if}& i\leq j, \\
0 &\textnormal{if} & i>j.
\end{array}\right.
\end{displaymath}

For simplicity we will prove the planar case
\begin{equation}
\label{planar}
\left\{\begin{array}{rcl}
\dot{x} &=& f(t,x,y)\\
\dot{y} &=& g(t,y).
\end{array}\right.
\end{equation}

As the jacobian matrix $JG$ of (\ref{planar}) is upper triangular, the assumption \textbf{(G3)} combined with Theorems 1 and 2 from \cite{Batelli} (see also p.540 from the same reference for details)
implies that
\begin{displaymath}
\Sigma^{+}\left[JG(t,\theta_{1}(t),\theta_{2}(t))\right]=\Sigma^{+}\left[\frac{\partial f}{\partial x}(t,\theta_{1}(t),\theta_{2}(t))\right]\cup \Sigma^{+}\left[\frac{\partial g}{\partial y}(t,\theta_{2}(t))\right],
\end{displaymath}
and \textbf{(G4)} allows to deduce that
\begin{equation}
\label{spectra}
\Sigma^{+}\left[\frac{\partial f}{\partial x}(t,\gamma_{1}(t),\gamma_{2}(t))\right]\subset (-\infty,0) \quad \textnormal{and} \quad
\Sigma^{+}\left[\frac{\partial g}{\partial y}(t,\gamma_{3}(t))\right]\subset (-\infty,0)
\end{equation}
for any set of measurable functions $t\mapsto \gamma_{i}(t)$ with $(i=1,2,3)$.

Let $t\to (x(t),y(t))$ be a solution of (\ref{planar}) passing through $(x_{0},y_{0})$ at $t=t_{0}$. Theorem \ref{N1}
combined with the second property of (\ref{spectra}) says that the origin is a globally uniformly asymptotically stable solution of the equation $\dot{y}=g(t,y)$ and then $t\mapsto y(t)$ satisfies
$|y(t)|\leq  \beta_{2}(|y_{0}|,t-t_{0})$ for any $t\geq t_{0}$.

Now, the component $t\mapsto x(t)$ is solution of the equation
$$
\dot{x}=f(t,x,y(t)) \quad \textnormal{with $x(t_{0})=x_{0}$}
$$
and we use again Theorem \ref{N1} combined with the first property of (\ref{spectra}) to deduce
$|x(t)|\leq \beta_{1}(|x_{0}|,t-t_{0})$ for any $t\geq t_{0}$. The properties of $\mathcal{KL}$ functions implies the existence of $\beta\in \mathcal{KL}$  such that
\begin{displaymath}
|x(t)|+|y(t)|\leq \beta(|x_{0}|+|y_{0}|,t-t_{0})
\end{displaymath}
and the global uniform stability of the trivial solution follows.
\end{proof}

\medskip

The following result shows a family of perturbed linear systems such that the conjecture is true for dimensions $n\geq 1$.

\begin{theorem}
Let the bounded and continuous linear system
\begin{equation}
\label{lin1}
\dot{x} = A(t)x
\end{equation} which has exponential dichotomy on $\mathbb{R}^{+}$ with constants $K \geq 1, \alpha > 0$ and projector $P = I$. If the nonlinear system
\begin{equation}
\label{nolin1}
\dot{x} = A(t)x+ f(t,x)
\end{equation}
 where $f\colon \mathbb{R}^{+} \times \mathbb{R}^n \to \mathbb{R}^n$ is continuous with respect to $t$, $C^1$ with respect to $x$, $f(t,x) = 0$  if and only if $x=0$ for all $t \in \mathbb{R}^{+}$ and for any measurable function $y\colon \mathbb{R}^{+}\to \mathbb{R}^{n}$ it follows that
\begin{equation}
\label{derivada}
 \displaystyle \sup_{t \in \mathbb{R}^{+}} \abs{Jf(t,y(t))} < \frac{\alpha}{4 K^2}.
 \end{equation}
 Then the system (\ref{nolin1}) satisfies hypothesis of  NMYC and the trivial solution is globally uniformly asymptotically stable.
\end{theorem}

\begin{proof}
First, note that (\ref{nolin1}) satisfies trivially \textbf{(G1)--(G3)}. We only need to prove that the linearization
\begin{equation}
\label{auxiliar}
    \dot{z} = [A(t) + Jf(t,y(t))]z
\end{equation}
  is such that $\Sigma^{+}(A + Jf)  \subset (-\infty, 0)$ for any measurable function $t\mapsto y(t)$ to deduce \textbf{(G4)}.

By (i)--(ii) of Remark \ref{observacion}, the system (\ref{lin1}) is uniformly asymptotically stable, which is equivalent to $\Sigma^{+}(A) \subset (-\infty, 0)$ and also equivalent to the exponential dichotomy on $\mathbb{R}^{+}$ with projector $P = I$. Now, the property (\ref{derivada}) combined with Proposition \ref{RCoppel} imply that the family (\ref{auxiliar}) has exponential dichotomy with the same projector $P = I$ and as consequence $\Sigma^{+}(A + Jf)  \subset (-\infty, 0)$.

  In order to prove that the trivial solution is uniformly asymptotically stable for the system (\ref{nolin1}), let  $x(t,t_0,\xi)$ be the solution of (\ref{nolin1}) passing through $\xi$ at $t = t_0.$ By using the mean value theorem combined with (\ref{derivada}) we have that

$$
\begin{array}{rcl}
|x(t,t_0,\xi)| & \leq & K e^{-\alpha(t-t_0)} |\xi| + \displaystyle K e^{-\alpha t} \int_{t_0}^{t} e^{\alpha s} |f(s,x(s,t_0,\xi))| \, ds\\\\
& \leq &  K e^{-\alpha(t-t_0)} |\xi| + \displaystyle K e^{-\alpha t} \int_{t_0}^{t} e^{\alpha s}  \frac{\alpha}{4K^2} |x(s,t_0, \xi)|\, ds.
\end{array}
$$

Notice that

$$e^{\alpha t} |x(t,t_0,\xi)| \leq K e^{\alpha t_0} |\xi| + \frac{\alpha}{4 K} \int_{t_0}^{t} e^{\alpha s}  |x(s,t_0, \xi)|\, ds,  $$
and Gronwall's lemma implies that

$$|x(t,t_0,\xi)| \leq  K e^{-\alpha (1 - \frac{1}{4 K}) (t- t_0)} |\xi|,$$
which implies the uniform asymptotical stability (see Definition \ref{UAS}) of the origin and the result follows.
\end{proof}

\section{A complementary result}

The following result links the stability of the origin in the autonomous and nonautonomous frameworks by showing a Hurwitz vector field $F\colon \mathbb{R}^{3}\to \mathbb{R}^{3}$, which  has the origin as global attractor. In spite that its proof is known \cite{CG}, we provide an alternative proof by using nonautonomous tools based in the exponential dichotomy spectrum.

\begin{theorem}
Let $ F = \lambda \, I + H\colon \mathbb{R}^{3}\to \mathbb{R}^{3}$ where
$$ H(x,y,z) =  g(z) \, (a(z) \, x + b(z) \, y) \, \left(\begin{array}{c}-b(z)\\
a(z)\\0
\end{array}\right)
\,  $$ with $ \lambda < 0 $, $ a,b,g \in \mathbb{R}[z] $ and $ F(0) = 0
$. Then there exists a related bidimensional nonautonomous linear system
\begin{equation}
\label{lineal-reducido}
\dot{u}=C(t)u
\end{equation}
which has exponential dichotomy with $P = I$ on $\mathbb{R}^+$ with
 $\Sigma_{ED}^{+}(C) = \{\lambda\}.$ Moreover, $F$ has the origin as global attractor.
\end{theorem}

\begin{proof}
 Note that $ (x(t),y(t),z(t)) $ is a solution of the differential system $ \dot{x} = F(x) $ if and only if $ z(t) = z_0 \, e^{\lambda t} $. As the case $z_{0}=0$ is straightforward, we will only consider the case  $z_0 \neq 0$.

Thus $ (x(t), y(t)) $ is a solution of the nonautonomous linear system (\ref{lineal-reducido}) with
$$ C(t)  = \left(
                  \begin{array}{cc}
                    \lambda - A(t)B(t)G(t) & -B(t)^2 G(t) \\
                    A(t)^2 G(t) &  \lambda + A(t)B(t)G(t)\\
                  \end{array}
                \right) $$
where $ (A,B,G)(t) = (a,b,g)(z_0 \, e^{\lambda t}) $.
Moreover, $C(t)$ can be written as $\lambda I+C_{0}(t)$:
$$
 \left(
                  \begin{array}{cc}
                    \lambda & 0  \\
                     0 & \lambda  \\
                  \end{array}
                \right)+  \left(
                  \begin{array}{cc}
                    - A(t)B(t)G(t) & -B(t)^2 G(t) \\
                    A(t)^2 G(t) &   A(t)B(t)G(t)\\
                  \end{array}
                \right),
$$
where $C_{0}(t)\to 0$ when $t\to +\infty$.

It is easy to see that the system $\dot{w} = \lambda I w$ has an exponential dichotomy on $\mathbb{R}^{+}$ with projector $P = I$ and its spectrum is $\Sigma_{ED}^{+} (\lambda I) = \{\lambda\}\subset (-\infty,0).$

We will prove that (\ref{lineal-reducido}) is uniformly asymptotically stable and $\Sigma_{ED}^{+}(C)=\{\lambda\}$. Firstly as $||C_{0}(t)|| \to 0$ as $t \to \infty,$ we use Proposition \ref{RWiggins} to obtain that (\ref{lineal-reducido}) has an exponential dichotomy on $[0,+\infty)$ with projector $P = I.$

It is remains to prove that $\Sigma_{ED}(C) = \{\lambda\}.$ Indeed we must study the following differential system for any $\gamma \in \mathbb{R}:$

\begin{equation}
\label{system}
\dot{x} =  \left(
                  \begin{array}{cc}
                    \lambda - \gamma & 0  \\
                     0 & \lambda - \gamma  \\
                  \end{array}
                \right)+  \left(
                  \begin{array}{cc}
                    - A(t)B(t)G(t) & -B(t)^2 G(t) \\
                    A(t)^2 G(t) &   A(t)B(t)G(t)\\
                  \end{array}
                \right)x.
\end{equation}

Notice that the system $\dot{x} = (\lambda - \gamma)x$ has exponential dichotomy on $\mathbb{R}^+$ with $P = I$ (resp. $P = 0$) when $\lambda < \gamma$ (resp. $\lambda > \gamma$). Thus, by using again the
roughness of exponential dichotomy we can deduce that the system (\ref{system}) has exponential dichotomy on $\mathbb{R}^+$ with $P = I$ (resp. $P = 0$) when $\lambda < \gamma$ (resp. $\lambda > \gamma$). Thus,
$\Sigma_{ED}^{+}(C) = \{\lambda\}$ or $\Sigma_{ED}(C) = \emptyset.$

We will verify that $\Sigma_{ED}^{+}(C)=\{\lambda\}$, otherwise $\Sigma_{ED}(C) = \emptyset$ which is equivalent to $\rho(C):=\mathbb{R}\setminus \Sigma_{ED}^{+}(C) = \mathbb{R}.$ This last statement combined with the ambivalence of the projector of exponential dichotomy of (\ref{system}) depending if $\lambda > \gamma$ (\emph{i.e}, $P=0$) or $\lambda < \gamma$ (\emph{i.e}, $P=I$) is a contradiction with the invariance of the rank's dimension of the projector $P_{\lambda-\gamma}$ in the connected components of the $\rho(C)$ described in the statement (iii) of Remark \ref{observacion} (for more details see \cite[Lemma 5.11]{Kloeden}. Therefore we have proved that $\Sigma_{ED}(C) = \{\lambda\}.$

Finally, as the system $\dot{x} = C(t)x$  has exponential dichotomy on $\mathbb{R^+}$ with projector $P = I$ (\textit{i.e} the origin is globally uniformly asymptotically stable) and the original
autonomous has solution $ z(t) = z_0 \, e^{\lambda t} ,$ we can deduce that the origin is a global attractor of the differential system  $ \dot{x} = F(x) .$

\end{proof}

\begin{remark}
\label{observacion2}
Inside the proof, we provide an example of the invariance of the exponential dichotomy spectrum by additive vanishing perturbations in the continuous framework: $\Sigma_{ED}^{+}(A)=\Sigma_{ED}^{+}(A+\tilde{B})$ for any matrix $\tilde{B}(t)\to 0$ when $t\to +\infty$. This property can be easily proved by using Proposition \ref{RWiggins}. In addition, the above result is inspired in the work of P\"otzsche and Russ \cite[Prop.8]{Russ}, which studied the invariance of the spectrum in the discrete case. To the best of our knowledge, there are no results in the continuous case.
\end{remark}

\end{document}